\documentclass[11pt]{amsart}
\usepackage{amsmath,amssymb,amsthm,amscd}
\usepackage{amsmath}
\usepackage{amssymb}
\usepackage{amssymb,amsmath,amsthm,amsfonts,latexsym,euscript}
\usepackage{wasysym}
\usepackage{fancyhdr}
\fancypagestyle{plain}
\usepackage{hyperref}

\newtheorem{lemma}{Lemma}

\newtheorem{theorem}{Theorem}

\usepackage{amssymb}
\usepackage[T1]{fontenc}
\usepackage[latin1]{inputenc}
\usepackage[usenames,dvipsnames]{xcolor}
\usepackage{graphicx}
\usepackage{setspace}
\usepackage{txfonts}
\usepackage{longtable}
\usepackage{blindtext}
\usepackage{ulem}
\usepackage[margin=1in]{geometry}

\AtEndDocument{{\footnotesize%
 \textsc{S. Blackwell, Department of Mathematics, Saint Louis University, St. Louis, MO., 63103} \par  
  \textit{E-mail address}, S.~Blackwell: \texttt{sblackw3@slu.edu} \par
  \addvspace{\medskipamount}
  \textsc{G. Durham, Department of Mathematics, University of Georgia, Athens, GA., 30602} \par  
  \textit{E-mail address}, G.~Durham: \texttt{gjdurham@uga.edu} \par
  \addvspace{\medskipamount}
  \textsc{K. Thompson, Department of Mathematics and Computer Science, Davidson College, Davidson, NC., 28036} \par
  \textit{E-mail address}, K. ~Thompson: \texttt{ktthompson@davidson.edu} \par
  \addvspace{\medskipamount}
  \textsc{T. Treece, Department of Mathematics, University of Georgia, Athens, GA., 30602} \par
  \textit{E-mail address}, T.~Treece: \texttt{treece@uga.edu}
}}

\title{A Generalization of Mordell to Ternary Quadratic Forms}
\author{Sarah Blackwell}
\author{Gabriel Durham}
\author{Katherine Thompson}
\author{Tiffany Treece}
\begin{document}
\maketitle
\begin{abstract}
Mordell in 1958 gave a new proof of the three squares theorem. We generalize those techniques to characterize the integers represented by the remaining six ``Ramanujan-Dickson ternaries'' as well as three other ternary forms.
\end{abstract}
\section{Introduction and Statement of Results}
Given a quadratic form, a natural question to ask concerns the values represented by the form. One of the invaluable applications of such analysis of ternary quadratic forms is in addressing the question in the \textit{quaternary} setting. In \cite{Ramanujan}, Ramanujan proved the existence of $55$ quaternary diagonal positive definite quadratic forms which represent all $n \in \mathbb N$ (though, of course, it was later discovered one of those forms failed to represent $n=15$. The remaining $54$ do, however, give the complete list). His method was that of escalation--a technique appearing later in more general quadratic forms results such as in Bhargava's proof of the $15$ Theorem \cite{Bhargava}. In creating his list of $55$ quaternary diagonal forms, Ramanujan produced the following list of ternary forms: 
\begin{center}
$x^2+y^2+z^2$, $x^2+y^2+2z^2$, $x^2+y^2+3z^2$, $x^2+2y^2+2z^2$, $x^2+2y^2+3z^2$, $x^2+2y^2+4z^2$, $x^2+2y^2+5z^2$.
\end{center}
Upon knowing which integers fail to be represented by each of these ternaries, Ramanujan was able to create the list of quaternary quadratic forms that represent all $n \in \mathbb N$. He made (correct) claims about the values represented and not represented by the seven ternaries; however, he included no proofs. Dickson \cite{Dickson} some years later provided what are commonly accepted as the first complete set of proofs, and these seven forms subsequently are referred to in the literature as the Ramanujan-Dickson ternaries. Dickson's methods are predominantly algebraic in flavor; in addition to standard congruential arguments (used mostly to state which values a form fails to represent) he used reduction techniques.\\
\\
Some thirty years later, Mordell published a new proof of the three squares theorem \cite{Mordell}. His argument differed from Dickson's in several crucial ways. First, Mordell avoided all reduction methods. Second, Mordell's algebraic manipulations of the forms were much more general than Dickson's; however, Mordell then needed to make use of a shortest-vector theorem of Gauss which in turn depends upon the determinant of the form. Subsequently, Mordell \textit{only} provided a proof of the three squares theorem. We know now that with significant alteration, the tools generalize. In particular, we generalize the method in such a way that the shortest-vector theorem (and hence any condition on the determinant) is unnecessary. We begin by presenting what we are calling ``Mordell-style proofs'' of the numbers represented by the remaining six Ramanujan-Dickson ternaries.
\begin{theorem}
\begin{itemize}\label{RD}
\item[(a)] A positive integer $m$ is represented by $x^2+y^2+2z^2$ if and only if $m \neq 4^k(16 \ell +14)$.
\item[(b)] A positive integer $m$ is represented by $x^2+y^2+3z^2$ if and only if $m \neq 9^k(9 \ell +6)$.
\item[(c)] A positive integer $m$ is represented by $x^2+2y^2+2z^2$ if and only if $m \neq 4^k (8 \ell +7)$.
\item[(d)] A positive integer $m$ is represented by $x^2+2y^2+3z^2$ if and only if $m \neq 4^k(16 \ell +10)$.
\item[(e)] A positive integer $m$ is represented by $x^2+2y^2+4z^2$ if and only if $m \neq 4^k (16 \ell+14)$.
\item[(f)] A positive integer $m$ is represented by $x^2+2y^2+5z^2$ if and only if $m \neq 25^k (25 \ell \pm 10)$.
\end{itemize}
\end{theorem}
The altered Mordell-style applies to more than just these cases. One could ask, for example, if the method applies to all integer matrix ternaries of specified determinant. We answer in the affirmative for determinants $5$ and $6$, providing in the process results which previously did not appear in the literature:
\begin{theorem}\label{D5}
\begin{itemize}
\item[(a)] A positive integer $m$ is represented by $x^2+y^2+5z^2$ if and only if $m \neq 4^k(8 \ell +3)$.
\item[(b)] A positive integer $m$ is represented by $x^2+2y^2+2yz+3z^2$ if and only if $m \neq 25^k (25 \ell \pm 5)$.
\end{itemize}
\end{theorem}
\begin{theorem}\label{D6}
 A positive integer $m$ is represented by $x^2+y^2+6z^2$ if and only if $m \neq 9^k(9 \ell +3)$.
\end{theorem}
Notes: 
\begin{itemize}
\item The other ternary of determinant $6$ is one of the Ramanujan-Dickson forms, so we do not list it again.
\item In 1931, Jones proved that the forms of a given genus collectively represent all positive integers not ruled out by certain congruence conditions \cite{Jones}. Since all of the forms we consider are alone in their genus, these forms represent exactly the integers they represent locally. The class number one requirement is a somewhat strict condition on a quadratic form and prevents large degrees of generalization. For instance, it has been well documented (see \cite{Kaplansky}, \cite{Kelley}) that $x^2+y^2+7z^2$ is not class number one and--more crucially--the integers represented by that form are still not understood.
\end{itemize}
The remainder of the paper is organized as follows. After a brief background section, we proceed to the proofs related to the Ramanujan-Dickson ternaries. Then we end with the proofs of representation for the forms of determinant $5$ and $6$.\\
\\
\textbf{Acknowledgements:} This research was supported by the National Science Foundation (DMS-1461189). We additionally thank Jeremy Rouse for providing helpful discussions and suggestions and Paul Pollack for his help in acquiring the necessary background literature.
\section{Background}
We provide a brief introduction to the definition and notation commonly seen in the study of quadratic forms.\\
\\
Let $n \in \mathbb N$. An \textbf{\textit{n}-ary integral quadratic form} is a homogeneous integral polynomial of degree two of the form $$Q(x) = Q(x_1,...,x_n) = \displaystyle\sum_{1 \leq i \leq j \leq n}a_{ij}x_ix_j \in \mathbb Z[x_1,...,x_n].$$
An equivalent way to represent quadratic forms is via symmetric matrices in $\mathcal M_n(\mathbb Q)$; that is, to each $n$-ary integral quadratic form there exists a unique symmetric matrix $A_Q \in \mathcal M_n(\mathbb Q)$ such that $$Q(x) = x^{t} A_Q x.$$
Under such a representation, all diagonal terms are integers while all off-diagonal entries are allowed to be (at worst) half integers. When $A_Q \in \mathcal M_n(\mathbb Z)$ (equivalently, when all cross-terms of $Q$ are even) we say that $Q$ is an \textbf{integer-matrix} form (equivalently, that $Q$ is \textbf{classically integral}).\\
\\
Given two $n$-ary quadratic forms $Q_1$ and $Q_2$ with respective matrices $A_{Q_1}, A_{Q_2}$ we say that $Q_1$ and $Q_2$ are \textbf{equivalent (over $\mathbb Z$)} if and only if there exists a matrix $M \in GL_n(\mathbb Z)$ such that $A_{Q_1} = MA_{Q_2}M^t$. \\
\\
Let $Q(x)$ be an $n$-ary integral quadratic form and let $m \in \mathbb Z$. We say that $Q$ \textbf{represents} $m$ if there exists $v \in \mathbb Z^n$ such that $Q(v)=m$. When $Q(v) > 0$ (resp., $\leq 0$) for all $\vec{0} \neq v \in \mathbb Z^n$, we say that $Q$ is \textbf{positive definite} (resp. \textbf{negative definite}).\\
\\
Henceforth, by ``form'' we mean ``positive definite ternary classically integral quadratic form.''\\
\\
We now outline the method of Mordell for proving representability by the sum of three squares. Let $m \in \mathbb N$ and suppose $f(x,y,z)$ is a classically integral form of determinant $1$ which represents $m$. Write $$mf(x,y,z)=(Ax+By+mz)^2+(ax^2+2hxy+by^2),$$ where $A,B,a,h,b \in \mathbb Z$ with $$A^2+a \equiv B^2 + b \equiv 2AB +2h \equiv 0 \pmod{m},$$ and where necessarily $ab-h^2 = m$. A geometry of numbers result of Gauss \cite[II.3.4]{Cassels} guarantees that $f$ represents a positive integer $E \leq \sqrt[3]{2D}$. As in the case of the sum of three squares $D=1$, this guarantees that $f$ represents $1$. As $f$ is presumed to be classically integral, this means $f$ is equivalent to a ternary of the form $X^2 \perp c_1Y^2+c_2YZ+c_3Z^2$ where the determinant of the positive definite binary in $Y$ and $Z$ must be one. The only choice up to equivalence is $c_1=c_3=1$ and $c_2=0$. Thus $f$ is equivalent to the sum of three squares. The remainder of Mordell's proof involves verifying that $A,B,a,h,b$ exist.\\
\\
The result of Gauss clearly becomes especially difficult to apply as the determinant of the form and subsequently the bound on $E$ increases. Moreover, as the determinant increases more forms will appear as possible candidates. We posit that this is why Mordell only used this result on the sum of three squares. However, this result of Gauss is not necessary. Instead of using a small vector argument to isolate the form of interest $f$ as Mordell did, we instead add conditions on $A,B,h,a$ and $b$ to force $f$ to represent values all other forms of the same determinant except. Since the values of the coefficients of any reduced form are bounded by the determinant of the form we can create a finite list of all classically integral forms of a fixed determinant. We refer to Nebe \cite{Nebe}, who provides a finite list of all forms of a given determinant, up to equivalence. For historical completion, we note that the tables appearing on Nebe's webpage were originally compiled in 1958 by Brandt and Intrau \cite{BI} using a method of bounding coefficients of reduced quadratic forms (see \cite{Cox} for such an argument in the binary case). \\
\\
Our general method of proof is as follows. We begin with the same construction of $mf(x,y,z)$ as Mordell. Towards generalization we first require that $f(x,y,z)$ has determinant $D$, which forces $ab-h^2 = Dm$. Like Mordell, we must have $A^2+a\equiv B^2+b\equiv 2AB+2h \equiv 0 \pmod{m}$. We set $B\equiv{0} \pmod{m}$ and $b\equiv{h}\equiv0 \pmod{Dm}$, noting that additionally $b=\tfrac{Dm+h^2}{a}$. This satisfies all requirements except $A^2+a\equiv{0} \pmod{m}$ and $ab-h^2 = Dm$. To show these in fact hold we use Dirichlet's theorem of primes in an arithmetic progression to construct $a$ such that $a\nmid m$ and $\left(\frac{-a}{p}\right) = 1$ for all odd primes $p \vert m$. This ensures that $A^2+a \equiv 0 \pmod{m}$ has a solution. Furthermore, since $-Dm\equiv{h^2} \pmod{a}$ we must place sufficient restrictions on $a$ such that $\left(\frac{-Dm}{a}\right) = 1$. Last, we choose $A$ and $a$ so that $\frac{A^2+a}{m}$ is an integer not represented by all but one form of determinant $D$. This allows us to isolate a particular $f(x,y,z)$. 

\section{The Ramanujan-Dickson Ternaries}
In this section, we provide proofs for the six Ramanujan-Dickson ternaries not considered by Mordell.
\subsection{The form {$x^2+y^2+2z^2$}}
\begin{lemma}
Let $m \equiv 14 \pmod{16}$. Then $m$ is not represented by $x^2+y^2+2z^2$.
\end{lemma}
\begin{proof}
This is a simple exercise left to the reader.
\end{proof}
\begin{lemma}
Let $m$ be an even integer. $m$ is not represented by $x^2+y^2+2z^2$ if and only if $4m$ is not represented by $x^2+y^2+2z^2$.
\end{lemma}
\begin{proof}
One direction is trivial. If $4m$ is represented, where $m$ is even, then there exist integers $x,y,z$ with $$x^2+y^2+2z^2 \equiv 0 \pmod{8}.$$ This forces $x \equiv y \equiv z \equiv 0 \pmod{2}$. Writing $x=2x_1, y=2y_1,z=2z_1$, we then see $$x_1^2+y_1^2+2z_1^2 =m.$$
\end{proof}
\begin{lemma}
If $m = 4^k (16 \ell + 14)$ for $k, \ell \in \mathbb{Z}$, then $m$ is not represented by $x^2+y^2+2z^2$.
\end{lemma}
\begin{proof}
This follows immediately from the previous two lemmas.
\end{proof}
We now show that $x^2+y^2+2z^2$ represents all integers $m \neq 4^k(16\ell+14)$ for $k,\ell \in \mathbb Z^{\geq 0}$.
Let $m \in \mathbb{Z}$, where $4 \nmid m$ and with $m \not\equiv 14 \pmod{16}$. Consider the ternary quadratic form $f(x,y,z)$ of determinant $D=2$ given by 
$$mf(x,y,z) = (Ax+By+mz)^2 + (ax^2+2hxy+by^2),$$
where necessarily $ab-h^2=2m$. By construction $f$ represents $m$.  As there is only one form of determinant $2$, namely $x^2+y^2+2z^2$, so long as $A,B,a,h,$ and $b$ exist with $$A^2+a \equiv B^2+b \equiv 2AB+2h \equiv 0 \pmod{m},$$ we have shown that $x^2+y^2+2z^2$ represents $m$.\\
\\
Take $b \equiv h \equiv 0 \pmod{2m}$ and $B \equiv 0 \pmod{m}$. For $a$ and $A$, we proceed by cases on $m$.
\begin{itemize}
\item[(Case 1)] $m \equiv 1,5,9,13 \pmod{16}$. We take $a$ to be a positive prime with $a \equiv 1 \pmod{8}$ and $(\tfrac{-a}{p})=1$ for all odd primes $p \vert m$. For such a prime $a$:
$$
1  =  \prod_{p \vert m} \left(\dfrac{-a}{p}\right)
   =  \prod_{p \vert m} \left(\dfrac{p}{a}\right)
   =  \left(\dfrac{-2m}{a}\right).
$$
\item[(Case 2)] $m \equiv 3,7,11,15 \pmod{16}$. Let $a \equiv 5 \pmod{8}$, and $(\tfrac{-a}{p})=1$ for all odd primes $p \vert m$. Then proceed in a fashion identical to the previous case. 
\item[(Case 3)] $m \equiv 2 \pmod{16}$. Here $m = 2m_1$ for $m_1 \equiv 1, 9 \pmod{16}$. We choose $a \equiv 1 \pmod{8}$ to be a prime satisfying $(\tfrac{-a}{p})=1$ for all odd primes $p \vert m_1$. Then:
$$
1  =  \prod_{p \vert m_1} \left(\dfrac{-a}{p}\right)
   =  \prod_{p \vert m_1} \left(\dfrac{p}{a}\right)
   =  \left(\dfrac{m_1}{a}\right) 
   =  \left(\dfrac{-2m}{a}\right).
$$
\item[(Case 4)] $m \equiv 10 \pmod{16}$. Here $m=2m_1$ for $m_1 \equiv 5, 13 \pmod{16}$. We choose $a \equiv 1 \pmod{8}$ to be a prime satisfying $(\tfrac{-a}{p})=1$ for all odd primes $p \vert m_1$. We then proceed similarly to the previous case.
\item[(Case 5)] $m \equiv 6 \pmod{16}$. Here we have $m=2m_1$ for $m_1 \equiv 3, 11 \pmod{16}$. We choose $a = 2a_1$ where $a_1 \equiv 1 \pmod{8}$ is a prime satisfying $(\tfrac{-a}{p}) =1$ for all odd $p \vert m_1$. This choice of $a$ satisfies:
$$
1  =  \prod_{p|m_1} \left(\dfrac{-a}{p}\right)
   =  \left(\dfrac{-2}{m_1}\right) \prod_{p|m_1} \left(\dfrac{a_1}{p}\right)
   =  \left(\dfrac{m_1}{a_1}\right) 
   =  \left(\dfrac{-2m}{a_1}\right).
$$
\end{itemize}
This completes the proof of Theorem \ref{RD}(a).

\subsection{The form {$x^2+y^2+3z^2$}}
\textbf{\\}
We begin by noting there are two determinant 3 forms: $Q_1:x^2+y^2+3z^2$ and $Q_2:x^2+2y^2+2yz+2z^2$.
\begin{lemma}\label{L4}
If $m$ is not represented by $Q_2$, then $4m$ is not represented by $Q_2$.
\end{lemma}
\begin{proof}
Assume $4m$ is represented by $Q_2$, so there exist $x,y,z \in \mathbb{Z}$ such that
$$0 \equiv x^2 + 2(y^2+yz+z^2) \pmod{4}.$$ This forces $x \equiv y \equiv z \equiv 0 \pmod{2}$. Writing $x = 2x_1$, $y = 2y_1$, and $z = 2z_1$ for $x_1,y_1,z_1 \in \mathbb Z$,
\begin{eqnarray*}
m & = & x_1^2 + 2y_1^2 + 2y_1z_1 + 2z_1^2.
\end{eqnarray*}
\end{proof}
\begin{lemma}
If $m=4^k(8\ell+5)$ for $k,\ell \in \mathbb{Z}$, then $m$ is not represented by $Q_2$.
\end{lemma}
\begin{proof}
We can assume $m=8\ell+5$ by the previous lemma. Suppose $m$ were represented by $Q_2$. Then $x$ is necessarily odd and 
\begin{eqnarray*}
y^2+yz+z^2 &\equiv& \pm 2 \pmod{8}.
\end{eqnarray*}
As there is no solution to the above equation $m$ is not represented by $Q_2$.
\end{proof}
\begin{lemma}
If $m$ is not represented by $Q_1$, then $9m$ is not represented by $Q_1$.
\end{lemma}
\begin{proof}
Assume $9m$ is represented by $Q_1$, so there exist $x,y,z \in \mathbb Z$ such that $$0 \equiv x^2+y^2+3z^2 \pmod{9}.$$This implies $x \equiv y \equiv z \equiv 0 \pmod{3}$. Writing $x=3x_1$, $y=3y_1$, and $z=3z_1$ for $x_1,y_1,z_1 \in \mathbb Z$ we have
$$m  =  x_1^2+y_1^2+3z_1^2.$$
\end{proof}
\begin{lemma}
If $m=9^k(9\ell+6)$ for $k,\ell \in \mathbb{Z}$, then $m$ is not represented by $Q_1$.
\end{lemma}
\begin{proof}
A computer search shows that if $m \equiv 6 \pmod{9}$ then $Q_1$ cannot represent $m$. Applying the previous lemma gives the result.
\end{proof}

We now show that that $x^2+y^2+3z^2$ represents all integers $m\neq 9^k(9\ell+6)$ for $k,\ell \in \mathbb Z^{\geq 0}$. Let $m \in \mathbb{Z}$, where $9 \nmid m$ and with $m \not\equiv 6 \pmod{9}$. Consider a ternary quadratic form $f(x,y,z)$ of determinant $D=3$ given by 
$$mf(x,y,z) = (Ax+By+mz)^2 + (ax^2+2hxy+by^2),$$
where necessarily $ab-h^2=3m$. By construction $f$ represents $m$. We must show that $A,B,a,b$ and $h$ exist with $$A^2+a \equiv B^2 +b \equiv 2AB+2h \equiv 0 \pmod{m}.$$ To satisfy the second and third congruences choose $b \equiv h \equiv 0 \pmod{3m}$ and $B \equiv 0 \pmod{m}$. To show that $f$ is equivalent to $Q_1$ and not $Q_2$, we moreover require that $\tfrac{A^2+a}{m} \equiv 5 \pmod{8}$.\\
To demonstrate the selection of $A$ and $a$, we now proceed by cases on $m$.
\begin{itemize}
\item[(Case 1)]
$m \equiv 1 \pmod{4}$
\begin{itemize}
\item[(Subcase 1)]
Consider $3 \nmid m$. Let $a \equiv 1 \pmod{24}$ be an odd prime with $a \nmid m$ and $(\tfrac{-a}{p})=1$ for all primes odd primes $p  \vert m$.  Then
$$
1  =  \prod_{p|m} \left(\dfrac{-a}{p}\right)
   =  \prod_{p|m} \left(\dfrac{p}{a}\right)
   =  \left(\dfrac{-3}{a}\right) \prod_{p|m} \left(\dfrac{p}{a}\right)
   =  \left(\dfrac{-3m}{a}\right).
$$
Note we can also ensure $\tfrac{A^2+a}{m} \equiv 5 \pmod{8}$. If $m \equiv 1 \pmod{8}$, then $A \equiv \pm2 \pmod{8}$. Else, choose $A^2 \equiv 0 \pmod{8}$.
\item[(Subcase 2)]
$m \equiv 21 \pmod{36}$. Let $m:=3m_1\equiv 7 \pmod{12}$ and let $a := 3a_1$ where $a_1 \equiv 7 \pmod{8}$ is an odd prime with $a_1 \nmid m$ and $(\tfrac{-a}{p})=1$ for all odd primes $p|m_1$. The rest of the proof mimics the previous subcase, including in showing $\tfrac{A^2+a}{m} \equiv 5 \pmod{8}$.
\end{itemize}
\item[(Case 2)]
$m \equiv 3 \pmod{4}$
\begin{itemize}
\item[(Subcase 1)]
Consider $3 \nmid m$. Let $a \equiv 7 \pmod{24}$ be an odd prime with $a \nmid m$ and that $(\tfrac{-a}{p})=1$ for all odd primes $p|m$. If $m \equiv 3 \pmod{8}$ then choose $A^2 \equiv 0 \pmod{8}$. Else, choose $A \equiv \pm2 \pmod{8}$.
\item[(Subcase 2)]
$m \equiv 3 \pmod{36}$. Let $m := 3m_1 \equiv 1 \pmod{12}$ and let $a := 3a_1$ where $a_1 \equiv 1 \pmod{8}$ is an odd prime with $a_1 \nmid m$ and $(\tfrac{-a}{p})=1$ for all odd primes $p|m_1$. Again, proceeding as before the claim holds, including the additional constraint that $\tfrac{A^2+a}{m} \equiv 5 \pmod{8}$.
\end{itemize}
\item[(Case 3)]
$m \equiv 2 \pmod{4}$
\begin{itemize}
\item[(Subcase 1)] $m \equiv 2,10 \pmod{16}$. Consider $3 \nmid m$. Let $2m_1 := m$ with $m_1 \equiv 1,5,9,13 \pmod{16}$.
Choose $a \equiv 1 \pmod{48}$ an odd prime with $a \nmid m$ so that $(\tfrac{-a}{p})=1$ for all odd primes $p|m_1$. Then
$$
1  =  \prod_{p|m_1} \left(\dfrac{-a}{p}\right) 
   =  \left(\dfrac{-3}{a}\right) \prod_{p|m_1} \left(\dfrac{a}{p}\right)
   =  \left(\dfrac{-3}{a}\right) \left(\dfrac{m_1}{a}\right) \\
   =  \left(\dfrac{-3m}{a}\right).
$$
Additionally, we choose $A^2 \equiv 1 \pmod{16}$ (when $m \equiv 10 \pmod{16}$) and $A^2 \equiv 9 \pmod{16}$ otherwise.
\item[(Subcase 2)] $m \equiv 6, 14 \pmod{16}$. Consider $3 \nmid m$. Let $2m_1 := m$ with $m_1 \equiv 3,7,11,15 \pmod{16}$.
Choose $a \equiv 13 \pmod{48}$ an odd prime with $a \nmid m$ so that $(\tfrac{-a}{p})=1$ for all odd primes $p|m_1$. Then
$$
1  =  \prod_{p|m_1} \left(\dfrac{-a}{p}\right)
   =  (-1) \prod_{p|m_1} \left(\dfrac{p}{a}\right)
   =  \left(\dfrac{2}{a}\right) \left(\dfrac{m_1}{a}\right)
   =  \left(\dfrac{-3m}{a}\right).
$$
Furthermore, since $m_1 \equiv 3 \pmod{4}$, $5m \equiv 6 \pmod{8}$. Hence we choose $A^2\equiv 1 \pmod{8}$.
\item[(Subcase 3)]
Let $m=6m_1$, where $m_1 \equiv 1,5,9,13 \pmod{16}$. Suppose first that $m_1 \equiv 1 \pmod{3}$. When $m_1 \equiv 1, 9 \pmod{16}$ choosing $a:=2a_1$ with $a_1 \equiv 1 \pmod{8}$ a prime $a_1 \nmid m$ with $(\tfrac{-a}{p})=1$ for all primes $p \vert m_1$ the result holds (including the existence of an $A$ ensuring $\tfrac{A^2+a}{m} \equiv 5 \pmod{8}$). When $m_1 \equiv 5, 13 \pmod{16}$ again we choose $a:=2a_1$; however, here we have $a_1 \equiv 3 \pmod{8}$ a prime $a_1 \nmid m$ with $(\tfrac{-a}{p})=1$ for all primes $p \vert m_1$. 
\item[(Subcase 4)] Let $m=6m_1$, where $m_1 \equiv 3,7,11,15 \pmod{16}$ and suppose that $m_1 \equiv 2 \pmod{3}$.Choose $a=3a_1$ where $a_1 \equiv 3 \pmod{8}$ is an odd prime with $a \nmid m_1$ and $(\tfrac{-a}{p})=1$ for all $p \vert m_1$. Then choosing $A^2 \equiv 1 \pmod{8}$ satisfies all conditions, including $\tfrac{A^2+a}{m} \equiv 5 \pmod{8}$.
\end{itemize}
\end{itemize}

Note that we did not consider any case where $m \equiv 0 \pmod4$ due to Lemma \ref{L4}. 
This completes the proof of \textit{Theorem 1(b)}.

\subsection{The form {$x^2+2y^2+2z^2$}}
\textbf{\\}
\noindent
There are two forms of determinant $4$ to consider: $Q_1:x^2+2y^2+2z^2$ and $Q_2:x^2+y^2+4z^2$.
\begin{lemma}
Let $m \in \mathbb{Z}$. If $m \equiv 3 \pmod4$, then $m$ is not represented by $Q_2$.
\end{lemma}
\begin{proof}
The proof is trivial, and is left to the reader.
\end{proof}
\begin{lemma}
If $m$ is not represented by $Q_1$, then $4m$ is not represented by $Q_1$.
\end{lemma}
\begin{proof}
Suppose that $4m$ is represented by $Q_1$. Then there exist $x,y,z \in \mathbb Z$ so that
\begin{eqnarray*}
0 & \equiv & x^2 + 2(y^2+z^2) \pmod{4}.
\end{eqnarray*}
This implies $x\equiv y \equiv z \equiv 0 \pmod{2}$. Writing $x=2x_1$, $y=2y_1$, $z=2z_1$ we have $$m = x_1^2+2y_1^2+2z_1^2.$$
\end{proof}
\begin{lemma}
If $m=4^k(8\ell+7)$ for $k,\ell \in \mathbb{Z}$, then m is not represented by $Q_1$.
\end{lemma}
\begin{proof}
If $m \equiv 7 \pmod{8}$, that $m$ is not represented by $Q_1$ is immediate. The previous lemma completes the proof of the claim.
\end{proof}

We now show that that $x^2+2y^2+2z^2$ represents all integers $m\neq 4^k(8\ell+7)$ for $k,\ell \geq \in \mathbb Z^{\geq 0}$. Let $m \in \mathbb{Z}$, where $4 \nmid m$ and with $m \not\equiv 7 \pmod{8}$. Consider a ternary quadratic form $f(x,y,z)$ of determinant $D=4$ given by
$$mf(x,y,z) = (Ax+By+mz)^2 + (ax^2+2hxy+by^2),$$
where necessarily $ab-h^2=4m$. By construction $f$ represents $m$. We must show that $A,B,a,b$ and $h$ exist with $$A^2+a \equiv B^2 +b \equiv 2AB+2h \equiv 0 \pmod{m}.$$ To satisfy the second and third congruences choose $b \equiv h \equiv 0 \pmod{4m}$ and $B \equiv 0 \pmod{m}$. To show that $f$ is equivalent to $Q_1$ and not $Q_2$, we moreover require that $\tfrac{A^2+a}{m} \equiv 3 \pmod{8}$. In order to select $a$ and $A$ we proceed by cases on $m$:
\begin{itemize}
\item[(Case 1)]
$m \equiv 1 \pmod{8}$. Choose $a=2a_1$, where $a_1 \equiv 1 \pmod{4}$ is an odd prime, $a_1 \nmid m$, and $\left(\frac{-a}{p}\right)=1$ for all odd primes $ p \vert m$. We see that:
$$
1  =  \prod_{p|m} \left(\dfrac{-a}{p}\right)
   =  \left(\dfrac{-2}{m}\right)\prod_{p|m} \left(\dfrac{a_1}{p}\right)
   =  \left(\dfrac{m}{a_1}\right)
   =  \left(\dfrac{-4m}{a_1}\right).
$$
Choosing $A$ odd gives the desired result.
\item[(Case 2)]
$m \equiv 5 \pmod8$. Choose $a=2a_1$, where $a_1 \equiv 3 \pmod{4}$ is an odd prime, $a_1 \nmid m$, and $\left(\frac{-a}{p}\right)=1$ for all odd primes $p \vert m$. The remainder of the proof, including $A$ odd, is identical to the previous case.
\item[(Case 3)]
$m=2m_1 \equiv 2 \pmod{8}$. Choose $a=2a_1$, where $a_1$ is an odd prime, $a_1 \nmid m$, and $\left(\frac{-a}{p}\right)=1$ for all odd primes $ p \vert m$. To select $a_1$ we proceed by cases on $m_1$:
\begin{itemize}
\item[(Subcase 1)] $m_1 \equiv 1 \pmod{8}$. Choose $a_1 \equiv 1 \pmod{8}$ and $A \equiv \pm 2 \pmod{8}$. Then the result follows, as
$$
1  =  \prod_{p|m_1} \left(\dfrac{-a}{p}\right)
   =  \left(\dfrac{-2}{m_1}\right)\prod_{p|m_1} \left(\dfrac{a_1}{p}\right) 
   =  \left(\dfrac{m_1}{a_1}\right) 
   =  \left(\dfrac{-4m}{a_1}\right).
$$
\item[(Subcase 2)] $m_1 \equiv 5 \pmod{8}$. Choose $a_1 \equiv 5 \pmod{8}$, $A \equiv \pm 2 \pmod{8}$ and proceed as in the previous subcase.
\end{itemize}
\item[(Case 4)]
$m=2m_1 \equiv 6 \pmod8$. Choose $a=2a_1$, where $a_1$ is an odd prime, $a_1 \nmid m$, and $\left(\frac{-a}{p}\right)=1$ for all odd primes $p \vert m$. As in Case $2$ we have two subcases:
\begin{itemize}
\item[(Subcase 1)] $m_1 \equiv 3 \pmod{8}$. Choose $a_1 \equiv 7 \pmod{8}$ and $A \equiv \pm 2 \pmod{8}$.
\item[(Subcase 2)] $m_1 \equiv 7 \pmod{8}$. Choose $a_1 \equiv 3 \pmod{8}$ and $A \equiv \pm 2 \pmod{8}$.
\end{itemize}

\item[(Case 5)]
$m \equiv 3 \pmod{8}$. Choose $a=8a_1$, where $a_1 \equiv 3 \pmod{4}$ is an odd prime, $a_1 \nmid m$, and $\left(\frac{-a}{p}\right) =1$ for all odd primes $ p \vert m$. Then
$$
1  =  \prod_{p|m} \left(\dfrac{-a}{p}\right)  
   =  \left(\dfrac{-2}{m}\right) \prod_{p|m} \left(\dfrac{a_1}{p}\right)
   =  \left(\dfrac{m}{a_1}\right)
   =  \left(\dfrac{-4m}{a_1}\right).
$$
Choosing $A$ odd gives the result.
\end{itemize}

This completes the proof of Theorem \ref{RD}(c).

\subsection{{The form {$x^2+2y^2+3z^2$}}}
\textbf{\\}
There are two quadratic forms of determinant 6: $Q_1:x^2+y^2+6z^2$ and $Q_2:x^2+2y^2+3z^2$.
\begin{lemma}
Let $m \in \mathbb{Z}$. If $m$ is not represented by $Q_1$, then $9m$ is not represented by $Q_1$.
\end{lemma}
\begin{proof}
Suppose $9m$ is represented by $Q_1$. Then $0 \equiv x^2+y^2+6z^2 \pmod{9}$. This implies $x \equiv y \equiv z \equiv 0 \pmod{3}$. Writing $x = 3x_1, y=3y_1, z=3z_1$, we have $m = x_1^2+y_1^2+6z_1^2$.
\end{proof}
\begin{lemma}
If $m=9^k(9\ell+3) \text{ for } k,\ell \in \mathbb{Z}$, then $m$ is not represented by $Q_1$.
\end{lemma}
\begin{proof}
If $m \equiv 3 \pmod{9}$, that $m$ is not represented by $Q_1$ is an easy proof left to the reader. The greater claim follows from the previous lemma.
\end{proof}
\begin{lemma}
Let $m$ be even. If $m$ is not represented by $Q_2$, then $4m$ is not represented by $Q_2$.
\end{lemma}
\begin{proof}
Suppose $4m$ is represented by $Q_2$. Then $0 \equiv x^2+2y^2+3z^2 \pmod{8}$. Then $x \equiv y \equiv z \equiv 0 \pmod{2}$ and $m$ is represented by $Q_2$.
\end{proof}
\begin{lemma}
If $m=4^k(16\ell+10) \text{ for } k,\ell \in \mathbb{Z}$, then $m$ is not represented by $Q_2$.
\end{lemma}
\begin{proof}
If $m = 16 \ell + 10$, that $m$ is not represented by $Q_2$ is an easy proof left to the reader. The greater claim follows from the previous lemma.
\end{proof}

We now show that that $Q_2$ represents all integers $m\neq 4^k(16\ell+10)$ for $k,\ell \in \mathbb Z^{\geq 0}$. Let $m \in \mathbb{Z}$, where $4 \nmid m$ and $m \not\equiv 10 \pmod{16}$.
Consider a ternary quadratic form $f(x,y,z)$ of determinant $D=6$ given by
$$mf(x,y,z) = (Ax+By+mz)^2 + (ax^2+2hxy+by^2),$$
where necessarily $ab-h^2=6m$. By construction $f$ represents $m$. 
Choose $b \equiv h \equiv 0 \pmod{6m}$. Considering
$$A^2+a \equiv B^2+b \equiv 2AB +2h \equiv 0 \pmod{m},$$
the second and third congruence conditions are satisfied if we choose $B \equiv 0 \pmod{m}$. In addition to satisfying the first congruence condition, we wish to force $\tfrac{A^2+a}{m} \equiv 3 \pmod{9}$ to ensure $f$ is equivalent to $Q_2$. We have the following cases:
\begin{itemize}
\item[(Case 1)]
$m \equiv 1\pmod{4}$. First suppose $3 \nmid m$. Choose $a \equiv 11\pmod{36}$ prime with $a \nmid m$ and $(\tfrac{-a}{p})=1$ for all primes $p \vert m$. Then
$$
1  =  \prod_{p|m} \left(\dfrac{-a}{p}\right)
   =  \left(\dfrac{m}{a}\right) 
   =  \left(\dfrac{-6m}{a}\right).
$$
Moreover, as $3m \equiv 3,15, 27 \pmod{36}$ then there will exist an $A$ such that $A^2 \equiv 3m-a \pmod{36}$. \\
When $3 \vert m$, a similar argument follows. Our restriction of $a \equiv 11 \pmod{36}$ implies that $a \equiv 2,11,20 \pmod{27}$. Regardless of the choice of $a$, there is an integer $A$ such that $A^2 +a \equiv 3m \pmod{27}$ holds.
\item[(Case 2)]
$m \equiv 3 \pmod{4}$. Again, first suppose $3 \nmid m$. Choose $a \equiv 41 \pmod{72}$ prime with $a \nmid m$ and $(\tfrac{-a}{p})=1$ for all $p \vert m$. Just as in the previous case, all necessary conditions are satisfied, including the existence on an $A$ such that $A^2 +a \equiv 3m \pmod{27}$. Similarly, the case where $3 \vert m$ adds no difficulty. Regardless of choice of $a$ (where here $a \equiv 5, 14, 23 \pmod{27}$) there is a solution $A$ to $A^2+a \equiv 3m \pmod{27}$.
\item[(Case 3)]
$m \equiv 2,6,14 \pmod{16}$. We write $m=2m_1$, with $m_1 \equiv 1,3,7 \pmod{8}$. We have two subcases:
\begin{itemize}
\item[(Subcase 1)] $m_1 \equiv 1 \pmod{8}$. Choose $a=2a_1$, where $a_1 \equiv 1 \pmod{36}$ is prime, with $a_1 \nmid m$ and $(\tfrac{-a}{p})=1$ for all primes $p \vert m_1$. Then
$$
1  =  \prod_{p|m_1} \left(\dfrac{-a}{p}\right)
   =  \left(\dfrac{m_1}{a_1}\right) 
   =  \left(\dfrac{2m}{a_1}\right) 
   =  \left(\dfrac{-6m}{a_1}\right).
$$
This gives $A^2 \equiv 6-2a_1, 18-2a_1, 30-2a_1 \pmod{36}$, which always has a solution. Moreover, regardless of $m \pmod{3}$ we have a solution to $A^2+a \equiv 3m \pmod{27}$, as the situation is identical to previous cases.
\item[(Subcase 2)] $m_1 \equiv{3,7} \pmod{8}$. Choose $a \equiv 5 \pmod{36}$ prime with $a \nmid m$ and $(\tfrac{-a}{p})=1$ for all $p \vert m_1$. The result follows in a manner similar to previous cases.
\end{itemize}
\end{itemize}
This completes the proof of Theorem \ref{RD}(d).

\subsection{The form {$x^2+2y^2+4z^2$}}
\textbf{\\}
There are four forms of determinant $8$: $Q_1: x^2+2y^2+4z^2$, $Q_2: x^2+y^2+8z^2$, $Q_3: x^2+3y^2+3z^2+2yz$ and $Q_4: 2x^2+2y^2+3z^2+2yz+2xz$.\\
\begin{lemma}
{Let $m\in\mathbb{Z}$. If $ m\equiv6 \pmod{16}$, then $m$ is not represented by $Q_2, Q_3,$ or $Q_4$.} 
\end{lemma}
\begin{proof} This is a simple exercise that is left to the reader. 
\end{proof}

\begin{lemma}
{Let $m\in\mathbb{Z}$. If $m\equiv14 \pmod{16}$ , then $m$ is not represented by $Q_1$.}
\end{lemma}
\begin{proof} Similarly, this is a simple exercise that is left to the reader.
\end{proof}

\begin{lemma}
{If $Q_1$ represents $4m$, then $m \not\equiv 14 \pmod{16}$.}
\end{lemma}
\begin{proof}
Suppose that $Q_1$ represents $4m$ with $m \equiv 14 \pmod{16}$. This means that $$8 \equiv x^2+2y^2+4z^2 \pmod{16}.$$
This forces $x \equiv y \equiv 0 \pmod{2}$. Writing $x=2x_1$, $y=2y_1$, this means $$m = x_1^2+2y_1^2+z^2.$$ Thus $m$ is represented by $X^2+Y^2+2Z^2$; however, that form does not represent any $m \equiv 14 \pmod{16}$ as shown in an earlier section.
\end{proof}

We now show that that $x^2+2y^2+4z^2$ represents all integers $m$ except $m=4^k(16\ell+14)$ for $k,\ell \geq 0$. Let $m \in \mathbb{Z}$, where $4 \nmid m$ and with $m \not\equiv 14 \pmod{16}$.  Consider the ternary quadratic form $f(x,y,z)$ of determinant $D=8$ given by $$mf(x,y,z)=(Ax+By+mz)^2+(ax^2+2hxy+by^2),$$ where necessarily $ab-h^2=8m$. We will show how to select $A, B, a, b,$ and $h$ so that $f(x,y,z)$ is equivalent to $Q_1$. Considering
$$A^2 + a \equiv  B^2 + b \equiv  2AB + 2h  \equiv  0 \pmod{m},$$
we begin by setting $h \equiv b \equiv 0 \pmod{8m}$ and $B \equiv 0 \pmod{m}$. The first condition will be satisfied if $\left(\tfrac{-a}{p}\right)=1$ for all odd primes $p \mid m$. Additionally, we will show $A$ and $a$ can be chosen so $\tfrac{A^2+a}{m} \equiv 6 \pmod{16}.$ We now proceed by cases on $m$:

\begin{itemize}
\item[(Case 1)] $m \equiv 1, 9 \pmod{16}$. Let $a=2a_1$ where $a_1 \equiv 3 \pmod{16}$ is an odd prime, $a_1 \nmid m$ with $(\tfrac{-a}{p})=1$ for all $p \vert m$. Then
$$
1  =  \prod_{p|m}\left(\dfrac{-a}{p}\right)
   =  \left(\dfrac{-2}{m}\right)\prod_{p|m}\left(\dfrac{a_1}{p}\right)
   =  \left(\dfrac{m}{a_1}\right)
   =  \left(\dfrac{-8m}{a_1}\right).
$$
Furthermore, choosing $A \equiv 0 \pmod{4}$ ensures $\tfrac{A^2+a}{m} \equiv 6 \pmod{16}$.

\item[(Case 2)] $m \equiv 3, 11 \pmod{16}$. Let $a=2a_1$ where $a_1 \equiv 1 \pmod{16}$ is an odd prime, $a_1 \nmid m$ with $(\tfrac{-a}{p})=1$ for all $p \vert m$. Setting $A \equiv 0 \pmod{4}$ allows us to mimic the previous case exactly.

\item[(Case 3)] $m \equiv 5, 13 \pmod{16}$. Let $a=2a_1$ where $a_1 \equiv 7 \pmod{16}$ is an odd prime, $a_1 \nmid m$ with $(\tfrac{-a}{p})=1$ for all primes $p \vert m$. Then set $A \equiv 0 \pmod{4}$ and proceed as before.

\item[(Case 4)] $m \equiv 7, 15 \pmod{16}$. Let $a=2a_1$ where $a_1 \equiv 5 \pmod{16}$ is an odd prime, $a_1 \nmid m$ and $(\tfrac{-a}{p})=1$ for all primes $p \vert m$. Once again, selecting $A \equiv 0 \pmod{4}$ gives a claim identical to the other cases.

\item[(Case 5)] $m=2m_1 \equiv 2 \pmod{16}$ with $m_1 \equiv 1 \pmod8$. Let $a=8a_1$ where $a_1 \equiv 1 \pmod{32}$ is an odd prime, $a_1 \nmid m$ and $(\tfrac{-a}{p})=1$ for all primes $p \vert m_1$. Then
$$
1  =  \prod_{p|m_1}\left(\dfrac{-a}{p}\right)
   =  \prod_{p|m_1}\left(\dfrac{a_1}{p}\right)
   =  \left(\dfrac{m_1}{a_1}\right)
   =  \left(\dfrac{-8m}{a_1}\right).
$$
Choose $A \equiv 2 \pmod{32}$. Then, $\tfrac{A^2+a}{m} \equiv 6 \pmod{16}$ and $A^2+a \equiv 12 \pmod{16}$. 

\item[(Case 6)] $m=2m_1 \equiv 6 \pmod{16}$ with $m_1 \equiv 3 \pmod8$. Let $a=32a_1$ where $a_1\equiv 1 \pmod{32}$ is a prime with $a_1 \nmid m$ and $(\tfrac{-a}{p})=1$ for all primes $p \vert m_1$. Choosing $A \equiv 2 \pmod{32}$, we then mimic the previous case.

\item[(Case 7)] $m=2m_1 \equiv 10 \pmod{16}$ where $m_1 \equiv 5 \pmod8$. Let $a=8a_1$ where $a_1 \equiv 7 \pmod{32}$ is an odd prime, $a_1 \nmid m$ and $(\tfrac{-a}{p})=1$ for all $p \vert m_1$. Again selecting $A \equiv 2 \pmod{32}$ allows us to refer to an earlier case.
\end{itemize}

This completes the proof of Theorem \ref{RD}(e)

\subsection{The form {$x^2+2y^2+5z^2$}} 
\textbf{\\}
We note that the forms of determinant $10$ are $Q_1 = x^2+2y^2+5z^2$, $Q_2 = 2x^2+2y^2+2xz+3z^2$ and $Q_3= x^2+y^2+10z^2$. 
\begin{lemma}
Let $m \equiv 6 \pmod{16}$. Then $m$ is not represented by $Q_2$ and $m$ is not represented by $Q_3$.
\end{lemma}
\begin{proof}
The proof is simple, and is left to the reader.
\end{proof}
\begin{lemma}
$25m$ is represented by $Q_1$ if and only if $m$ is represented by $Q_1$.
\end{lemma}
\begin{proof}
One direction is trivial. Suppose $25m$ is represented by $x^2+2y^2+5y^2$. Then $x^2+2y^2+5y^2\equiv{0} \pmod {5}$ so $x^2\equiv{-2y^2} \pmod5$, which implies $x,y\equiv0 \pmod 5$. We now have $0+0+5z^2\equiv0 \pmod{25}$, so $5|z$. Substituting $z=5z_1$, $x=5z_1$ and $y=5y_1$ we see $$m = x_1^2+y_1^2+5z_1^2.$$ 
\end{proof}
\begin{lemma}
If $m = 25^k (25 \ell \pm 10)$, for $k, \ell \in \mathbb Z$, then $m$ is not represented by $Q_1$.
\end{lemma}
\begin{proof}
When $m \equiv \pm 10 \pmod{25}$ this claim is trivial. Otherwise, the result follows from the previous lemma.
\end{proof}
Now let $m \in \mathbb Z$, where $25 \nmid m$ and with $m \neq 25^k (25 \ell \pm 10)$. 
Consider the ternary quadratic form $f(x,y,z)$ of determinant $D=10$ given by:
\begin{center}
$mf(x,y,z)= (Ax+By+mz)^2+ax^2+2hxy+by^2,$
\end{center}
where necessarily $ab-h^2 = 10m$. Choose $h,b$ such that $h\equiv{0}\equiv{b} \pmod {10m}$. We also wish to choose $A,B,a,h,b$ so that 
$$A^2+a \equiv B^2 +b \equiv 2AB+2h \equiv 0 \pmod{m},$$ 
where $\tfrac{A^2+a}{m} \equiv 6 \pmod{16}.$  This will guarantee $m$ is represented by $x^2+2y^2+5z^2$. Furthermore, let $B \equiv 0 \pmod{m}$. We see that $B^2 +b \equiv 2AB+2h \equiv 0 \pmod{m}$. Now we must choose $A$ and $a$ so that $A^2+a\equiv0 \pmod{m}$ and $\tfrac{A^2+a}{m} \equiv 6 \pmod{16}$. We proceed by cases:
\begin{itemize}
\item[(Case 1)] $m$ odd with $5\nmid m$. Let $a$ be an odd prime with $a\nmid m$ and $\left(\frac{-a}{p}\right) = 1$ for all primes $ p|m$. We have two subcases:
\begin{itemize}
\item[(Subcase 1)] $m\equiv1 \pmod 4$. Let $a\equiv{13,37} \pmod {400}$ be prime. This choice of $a$ satisfies: 
$$
1  =  \prod_{p|m}{\left(\frac{-a}{p}\right)}
   =  \left(\frac{m}{a}\right)
   =  \left(\frac{-2}{a}\right)\left(\frac{a}{5}\right)\left(\frac{m}{a}\right)
   =  \left(\frac{-10m}{a}\right).
$$
We now consider four further subcases: $m\equiv{1,9,13,17} \pmod{20}$. For the remainder of the proof, we refer the reader to Table 1; this shows the complete list of $A$ values (dependent upon $m$) which we yield $\frac{A^2+a}{m} \equiv{E} \pmod{400}$ where $E\in\mathbb{Z}, E > 0$, and $E\equiv6 \pmod{16}$.
\item[(Subcase 2)] $m\equiv3 \pmod 4$. We will add the constraint that $a\equiv17 \pmod {400}$. We again refer to Table 1 to consider the subsequent four subcases.
\end{itemize}
\item[(Case 2)] $m$ even and $5\nmid m$. Choose $a\equiv{3} \pmod {800}$ to be prime with $a \nmid m$ and $\left(\frac{-a}{p}\right) = 1$ $\forall p|m_1$. Again, we have two subcases:
\begin{itemize}
\item[(Subcase 1)]
Let $m=2m_1\equiv2\pmod8$ where $m_1 \equiv 1 \pmod{4}$. Our choice of $a$ satisfies:
$$
1  =  \prod_{p|m_1}{\left(\frac{-a}{p}\right)}
   =  \left(\frac{m_1}{a}\right)
   =  \left(\frac{2m}{a}\right)
   =  \left(\frac{-10m}{a}\right).
$$
As we have done previously, we now consider four further subcases, $m\equiv{2,18,26,34} \pmod{40}$ and refer the reader to Table 1 for the choices of $A$.
\item[(Subcase 2)] Let $m=2m_1\equiv6 \pmod8$ where $m_1 \equiv 3 \pmod 4$. 
Again, we refer the reader to Table 1 for the four following subcases and their respective choices of $A$.
\end{itemize}

\item[(Case 3)] $m$ odd and $5 \vert m$. Write $m=5m_1$ where necessarily $m_1\equiv{1,4} \pmod{5}$. Now we choose $a = 5a_1$ where $a_1$ is an odd prime with $a_1\nmid m$ and $(\tfrac{-a}{p}) = 1$ for all $p \vert m_1$. We now have two subcases: 
\begin{itemize}
\item[(Subcase 1)]
Let $m\equiv5 \pmod {20}$ and $m_1\equiv1 \pmod4$. We choose $a_1\equiv1 \pmod{400}$ so that
$$
1  =  \prod_{p|m_1}{\left(\frac{-a}{p}\right)}
   =  \left(\frac{m_1}{a_1}\right)
   =  \left(\frac{-2}{a}\right)\left(\frac{5m}{a}\right)
   =  \left(\frac{-10m}{a}\right).$$
As is frequently the case, we now consider two further subcases: $m\equiv{5,45} \pmod{100}$. Moreover we refer the reader to Table 1 for the choices of $A$. 
\item[(Subcase 2)]
Let $m\equiv15 \pmod {20}$ and $m_1\equiv 3 \pmod 4$. Choose $a_1\equiv13 \pmod{400}$, and proceed to Table 1 for both subcases and choices of $A$.
\end{itemize}
\item[(Case 4)] $m$ even and $5 \vert m$. Then $m=10m_1$ where $m_1\equiv{2,3} \pmod{5}$. We set $a =5a_1$, where $a_1 \equiv 7 \pmod{800}$ is a prime with $a_1 \nmid m$ and with $1 = (\tfrac{-a}{p})$ for all $p \vert m_1$. We have two subcases:
\begin{itemize}
\item[(Subcase 1)] $m\equiv10 \pmod {40}$. Thus $m_1\equiv{1} \pmod{4}$ and as a result:
$$
1  =  \prod_{p|m_1}{\left(\frac{-a}{p}\right)}
   =  (-1) \left(\frac{m_1}{a_1}\right)
   =  \left(\frac{-1}{a_1}\right)\left(\frac{m_1}{a_1}\right)
   =  \left(\frac{-10m}{a}\right).
$$
As we have done previously, we consider two further subcases: $m\equiv{130,170} \pmod{200}$ and refer the reader to Table 1.
\item[(Subcase 2)] $m\equiv30 \pmod {40}$ and thus $m_1\equiv3 \pmod4$. Once again, we have two subcases $(m \equiv 30,70 \pmod{200})$ and refer the reader to Table 1.

\end{itemize}
\end{itemize}

Note that we did not consider any case where $m\equiv0 \pmod4$, i.e. $m = 4^k m_1$ for some $k,m_1\in\mathbb{Z}$ where $4\nmid m_1$. This omission is acceptable since, if $Q_1$ represents $m_1$, then $Q_1$ clearly represents $m$ as well. If $Q_1$ does not represent $m_1$ then $m_1\equiv{\pm 10} \pmod{25}$ so $m\equiv{\pm 10} \pmod{25}$ and $Q_1$ does not represent $m$.\\


\begin{longtable}{|c|c|c|c|}
\caption[Data For Subcases]{Data for Subcases} \label{grid_mlmmh} \\
\hline
(Case, Subcase) & $m$ & $A^2:=Em-a$ & Values for $A^2$ \\ 
\hline
$(1,1)$ & $1\pmod{20}$ & $22m-13$ & $9, 49, 89, 129, 169, 209, 249, 289, 329, 369 \pmod{400}$ \\
$(1,1)$ & $9\pmod{20}$ & $22m-37$ & $1, 41, 81, 121, 161, 201, 241, 281, 321, 361 \pmod{400}$ \\
$(1,1)$ & $13\pmod{20}$ & $22m-37$ & $9, 49, 89, 129, 169, 209, 249, 289, 329, 369 \pmod{400}$ \\
$(1,1)$ & $17\pmod{20}$ & $22m-37$ & $1, 41, 81, 121, 161, 201, 241, 281, 321, 361 \pmod{400}$ \\
$(1,2)$ & $3\pmod{20}$ & $6m-17$ & $1, 41, 81, 121, 161, 201, 241, 281, 321, 361 \pmod{400}$ \\
$(1,2)$ & $7\pmod{20}$ & $6m-17$ & $ 9, 49, 89, 129, 169, 209, 249, 289, 329, 369 \pmod{400}$ \\
$(1,2)$ & $11\pmod{20}$ & $6m-17$ & $ 9, 49, 89, 129, 169, 209, 249, 289, 329, 369 \pmod{400}$  \\
$(1,2)$ & $19\pmod{20}$ & $22m-17$ & $1, 41, 81, 121, 161, 201, 241, 281, 321, 361 \pmod{400}$ \\
\hline
\hline
$(2,1)$ & $2\pmod{40}$ & $6m-3$ & $9, 89, 169, 249, 329, 409, 489, 569, 649, 729 \pmod{800}$ \\
$(2,1)$ & $18\pmod{40}$ & $38m-3$ & $41, 121, 201,281, 361,  441, 521, 601, 681, 761 \pmod{800}$ \\
$(2,1)$ & $26\pmod{40}$ & $22m-3$ & $9, 89, 169, 249, 329, 409, 489, 569, 649, 729 \pmod{800}$ \\
$(2,1)$ & $34\pmod{40}$ & $6m-3$ & $41, 121, 201,281, 361,  441, 521, 601, 681, 761 \pmod{800}$ \\
$(2,2)$ & $6\pmod{40}$ & $22m-3$ & $49, 129, 209, 289, 369, 449, 529, 609, 689, 769 \pmod{800}$ \\
$(2,2)$ & $14\pmod{40}$ & $38m-3$ & $49,129,209, 289, 369, 449, 529, 609, 689, 769 \pmod{800}$ \\
$(2,2)$ & $22\pmod{40}$ & $6m-3$ & $49, 129, 209, 289, 369, 449, 529, 609, 689, 769 \pmod{800}$ \\
$(2,2)$ & $38\pmod{40}$ & $38m-3$ & $1,81,161,241,321, 401,481,561,641,721 \pmod{800}$ \\
\hline
\hline
$(3,1)$ & $5\pmod{100}$ & $6m-5$ & $25, 225 \pmod{400}$ \\
$(3,1)$ & $45\pmod{100}$ & $54m-5$ & $25, 225 \pmod{400}$ \\
$(3,2)$ & $55\pmod{100}$ & $38m-65$ & $25, 225 \pmod{400}$ \\
$(3,2)$ & $95\pmod{100}$ & $22m-65$ & $25, 225 \pmod{400}$ \\
\hline
\hline
$(4,1)$ & $130\pmod{200}$ & $22m-35$ & $25, 425 \pmod{800}$ \\
$(4,1)$ & $170\pmod{200}$ & $38m-35$ & $25, 425 \pmod{800}$ \\
$(4,2)$ & $30\pmod{200}$ & $22m-35$ & $225, 625 \pmod{800}$ \\
$(4,2)$ & $70\pmod{200}$ & $38m-35$ & $225, 625 \pmod{800}$ \\

\hline

\end{longtable}

This completes the proof of Theorem \ref{RD}.

\section{Forms of Determinant Five}
There are two quadratic forms of determinant 5: $Q_1: x^2+y^2+5z^2$ and $Q_2: x^2+2y^2+2yz+3z^2$.
\begin{lemma}
If $4m$ is represented by $Q_1$, then $m$ is represented by $Q_1$.
\end{lemma}
\begin{proof}
Suppose $\exists x,y,z\in{\mathbb{Z}}$ such that $4m = x^2+y^2+5z^2$. Thus $x^2+y^2+5z^2\equiv0 \pmod 4$. This implies $x \equiv y \equiv z \equiv 0 \pmod{4}$. Let $ x = 2x_1, y = 2y_1, z = 2z_1$. Then $m = x_1^2+y_1^2+5z_1^2$.
\end{proof}
\begin{lemma}
If $m = 4^k (8 \ell +3)$ for $k, \ell \in \mathbb{Z}$, then $m$ is not represented by $Q_1$.
\end{lemma}
\begin{proof}
If $m \equiv 3 \pmod{8}$, that $m$ is not represented by $Q_1$ is a simple proof left to the reader. Otherwise, the claim follows from the previous lemma.
\end{proof}
\begin{lemma}
If $25m$ is represented by $Q_2$, then $m$ is represented by $Q_2$.
\end{lemma}
\begin{proof}
Suppose $25m$ is represented by $Q_2$. A computer search shows that this requires $x \equiv y \equiv z \equiv 0 \pmod{5}$. Substituting $x=5x_1, y=5y_1, z=5z_1$, we have $m= x_1^2+2y_1^2+2y_1z_1+3z_1^2$.
\end{proof}
\begin{lemma}
If $m \equiv 25^k (25 \ell \pm 5)$, then $m$ is not represented by $Q_2$.
\end{lemma}
\begin{proof}
If $m \equiv \pm 5 \pmod{25}$, that $m$ is not represented by $Q_2$ is a simple proof left to the reader. Otherwise, the claim follows from the previous lemma.
\end{proof}

\subsection{The form {$x^2+y^2+5z^2$}}
\textbf{\\}
We now show that that $x^2+y^2+5z^2$ represents all integers $m\neq 4^k(8\ell+3)$ for $k,\ell \in \mathbb Z^{\geq 0}$. Let $m \in \mathbb Z$, where $4 \nmid m$ and with $m \neq 8 \ell+3$.  
As we have done previously, consider the form $f(x,y,z)$ of determinant $D=5$ given by $$mf(x,y,z)= (Ax+By+mz)^2+(ax^2+2hxy+by^2).$$ We see that $ab-h^2 = 5m$. Choose $h,b$ such that $h\equiv{0}\equiv{b} \pmod {5m}$. We need
$$A^2+a \equiv B^2 + b \equiv 2AB + 2h \equiv 0 \pmod{m}.$$
The second two conditions are met if $B\equiv0 \pmod m$ and the first condition is met if we choose an $a$ such that $\left(\frac{-a}{p}\right)=1$ $\forall$ odd primes $p$ where $p \vert m$. Additionally, we will force $\tfrac{A^2+a}{m} \equiv 5 \pmod{25}$ to ensure $m$ is represented by $x^2+y^2+5z^2$. To choose $a$ appropriately, we proceed by cases on $m$.
\begin{itemize}
\item[(Case 1)]$m\equiv{1,2,5,6} \pmod 8$. Let $a$ be an odd prime with $a\nmid m$ and $\left(\frac{-a}{p}\right) = 1$ for all odd primes $p \vert m$. We now consider three subcases:
\begin{itemize}
\item[(Subcase 1)] $m \equiv 1,5 \pmod{8}$. Then let $a \equiv 1 \pmod{1000}$. We see that:
$$
1  =  \prod_{p|m}{\left(\frac{-a}{p}\right)}
   =  \prod_{p|m}{\left(\frac{a}{p}\right)}
   =  \left(\frac{m}{a}\right)
   =  \left(\frac{-5m}{a}\right).
$$
Note that here we can also allow $\tfrac{A^2+a}{m} \equiv 5 \pmod{1000}$, as $5m-1$ 
is always a perfect square $\pmod{1000}$. The result follows.
\item[(Subcase 2)] $m=2m_1\equiv2 \pmod 8$. Then $m_1 \equiv 1 \pmod{4}$. Let $a \equiv 1 \pmod{1000}$. Then we return to an identical situation to Subcase 1, including with the constraint $\tfrac{A^2+a}{m} \equiv 5 \pmod{1000}$.
\item[(Subcase 3)] $m=2m_1 \equiv6 \pmod 8$. Then $m_1 \equiv 3 \pmod{4}$. Let $a \equiv 11 \pmod{1000}$. Here we force $\frac{A^2+a}{m}\equiv{30} \pmod{1000}$, yet otherwise our work mimics that of earlier cases. 
\end{itemize}
\item[(Case 2)] $m\equiv7 \pmod 8$. We choose now $a= 2a_1$, where $a_1 \equiv 13 \pmod{1000}$ is an odd prime such that $a_1\nmid m$ and $(\tfrac{-a}{p})=1$ for all odd primes $p \vert m$. Then
$$
1  =  \prod_{p|m}{\left(\frac{-a}{p}\right)}
   =  (-1) \prod_{p|m}{\left(\frac{p}{a_1}\right)}
   =  (-1) \left(\frac{m}{a_1}\right)
   =  \left(\frac{-5m}{a_1}\right).
$$
Then all conditions are met, including $\tfrac{A^2+a}{m} \equiv 5 \pmod{1000}$, as $5m-26$ is always a perfect square $\pmod{1000}$. This completes the proof of Theorem \ref{D5}(a).
\end{itemize}

\subsection{{The form {$x^2+2y^2+2yz+3z^2$}}}
\textbf{\\}
We now show that that $x^2+2y^2+2yz+3z^2$ represents all integers $m \neq 25^k(25\ell \pm 5)$ for $k,\ell \in \mathbb Z^{\geq 0}$. Let $m \in \mathbb{Z}$, where $25 \nmid m$, and with $m \not\equiv \pm 5 \pmod{25}$. Consider the form $f(x,y,z)$ of determinant $D=5$ given by $$mf(x,y,z)= (Ax+By+mz)^2+(ax^2+2hxy+by^2).$$ We see that $ab-h^2 = 5m$. Choose $h,b$ such that $h\equiv{0}\equiv{b} \pmod {5m}$. 
We need
$$A^2+a \equiv B^2 + b \equiv 2AB + 2h \equiv 0 \pmod{m}.$$
The second two conditions are met if $B\equiv0 \pmod m$ and the first condition is met if we choose an $a$ such that $\left(\frac{-a}{p}\right)=1$ $\forall$ odd primes $p$ where $p \vert m$. Additionally, we will force $\tfrac{A^2+a}{m} \equiv 3 \pmod{8}$. We now proceed by cases on $m$. 

\begin{itemize}
\item[(Case 1)] $m$ odd and $5 \nmid m$. We create a list $E$ of congruence classes $\pmod{200}$ which are not $0 \pmod 5$ but which are $3 \pmod{8}$: $$E= \{ 3,11,19,27,43,51,59,67,75,83,91,99,107,123,131,139,147,163,171,179,187 \}.$$ 
\begin{itemize}
\item[(Subcase 1)] $m \equiv 1 \pmod{8}$. We call $M_{C}$ the possible values of $m \pmod{200}$:
$$M_{C}= \{  1,9,17,33,41,49,57,73,81,89,97,113,121,129,137,153,161,169,177,185,193\}.$$
By a computer check, for any two $m_1, m_2 \in M_{C}$, $\{ m_1 \cdot e, e \in E \} = \{ m_2 \cdot e, e \in E \}$. Let $E_m$ be the element of $E$ such that for fixed $m \in M_{C}$, $E_m \cdot m \equiv 3 \pmod{200}$. Let $a \nmid m$ be prime with $a \equiv 3 \pmod{200}$ and $\left(\frac{-a}{p}\right)=1$ for all $p \vert m$.Then
$$
1  =  \prod_{p|m} \left( \dfrac{-a}{p} \right)
   =  \left( \dfrac{m}{a} \right)
   =  \left( \dfrac{-5m}{a} \right).
$$
Last, choosing $A$ so that $A^2 \equiv 0 \pmod{200}$ completes the proof in this case.
\item[(Subcase 2)] $m \equiv 3 \pmod{8}$. Again, call $M_C$ the possible values of $m \pmod{200}$, and note that here $E=M_C$. We verify that for each $m \in M_{C}$, $\{ m \cdot e, e \in E \}=E$. Let $E_m$ then be the element of $E$ such that $E_m \cdot m \equiv 17 \pmod{200}$. Let $a \nmid m$ be a prime such that $a \equiv 17 \pmod{200}$, and $\left(\frac{-a}{p}\right)=1$ for all $p \vert m$. Choosing $A^2 \equiv 0 \pmod{200}$ again completes the proof.
\item[(Subcase 3)] $m \equiv 5 \pmod{8}$. Here we have
$$M_{C}= \{ 13,21,29,37,53,61,69,77,93,101,109,117,133,141,149,157,165,173,181,189,197 \}.$$ For any two $m_1, m_2 \in M_{C}$, $\{ m_1 \cdot e, e \in E \} = \{ m_2 \cdot e, e \in E \}$. Let $E_m$ be the element of $E$ such that for fixed $m \in M_{C}$, $E_m \cdot m = 7\pmod{200}$. Let $a \nmid m$ be a prime such that $a \equiv 7 \pmod{200}$, and $\left(\frac{-a}{p}\right)=1$ for all $p \vert m$. Once again choosing $A^2 \equiv 0 \pmod{200}$ gives the desired results.
\item[(Subcase 4)] $m \equiv 7 \pmod{8}$. Then
$$M_{C}= \{ 7,23,31,39,47,63,71,79,87,103,111,119,127,143,151,159,167,175,183,191,199 \}$$ Once again for any two $m_1, m_2 \in M_{C}$, $\{ m_1 \cdot e, e \in E \} = \{ m_2 \cdot e, e \in E \}$. Let $E_m$ be the element of $E$ such that for fixed $m \in M_{candidates}$, $E_m \cdot m = 13 \pmod{200}$. Let $a \nmid m$ be prime with $a \equiv 7 \pmod{200}$, and $\left(\frac{-a}{p}\right)=1$ for all $p \vert m$. With $A$ chosen such that $A^2 \equiv 0 \pmod{200}$ the claim follows.
\end{itemize}

\item[(Case 2)] $m$ even and $5 \nmid m$. We have the following cases:
\begin{itemize}
\item[(Subcase 1)] $m=2m_1 \equiv 2 \pmod{16}$ where $m_1 \equiv 1\pmod{8}$. We now consider $M_C$, the collection of possible values of $m \pmod{400}$. Thus $$M_{C} = \{ 2,18,34,66,82,98,114,146,162,178,194,226,242,258,274,306,322,338,354,386 \}$$ and we similarly extend the definition of $E$. Note that here $$m \cdot E = \{ 6,22,38,54,86,102,118,134,166,182,198,214,246,262,278,294,326,342,358,374 \}.$$
Let $E_m\in E$ be the element satisfying $E_m \cdot m \equiv 6 \pmod{400}$. Then choose $a=2a_1$ where $a_1 \equiv 1 \pmod{400}$ is prime, with $a_1 \nmid m$ and $\left(\frac{-a}{p}\right)=1$ for all odd $p \vert m_1$. We see that:
$$
1  =  \prod_{p|m_1} \left( \dfrac{-a}{p} \right)
   =  \left( \dfrac{-2}{m_1} \right)\prod_{p|m_1} \left( \dfrac{a_1}{p} \right)
   =  \left( \dfrac{m_1}{a_1} \right)
   =  \left( \dfrac{-5m}{a_1} \right).
$$
Last choose $A$ so that $A^2 \equiv 4 \pmod{400}$. 
\item[(Subcase 2)] $m=2m_1 \equiv 6 \pmod{16}$ where $m_1 \equiv 3 \pmod{8}$.
$$M_{C}= \{ 6,22,38,54,86,102,118,134,166,182,198,214,246,262,278,294,326,342,358,374\} \textrm{ and }$$
$$m \cdot E = \{ 2,18,34,66,82,98,114,146,162,178,194,226,242,258,274,306,322,338,354,386 \}.$$
Let $E_m$ be the element of $E$ such that $E_m \cdot m \equiv 2 \pmod{400}$. Then choose $a=2a_1$ where $a_1 \equiv 1 \pmod{400}$, $a_1 \nmid m$ is prime with $\left(\frac{-a}{p}\right)=1$ for all odd primes $p \vert m_1$. Then choose $A$ with $A^2 \equiv 0 \pmod{400}$ and the result follows.

\item[(Subcase 3)] $m=2m_1 \equiv 10 \pmod{16}$ where $m_1 \equiv 5 \pmod{8}$. Then for $m \in M_C$, we have $$m\cdot E = \{ 14,46,62,78,94,126,142,158,174,206,222,238,254,286,302,318,334,366,382,398\}.$$
Let $E_m \in E$ be such that $E_m \cdot m \equiv 62 \pmod{400}$. Then choose $a=2a_1$ where $a_1 \equiv 29 \pmod{400}$ is a prime with $a_1 \nmid m$ and $(\tfrac{-a}{p})=1$ for all odd $p \vert m_1$. Then choose $A^2 \equiv 4 \pmod{400}$ and the result follows.
\item[(Subcase 4)] $m=2m_1 \equiv 14 \pmod{16}$ where $m_1 \equiv 7 \pmod{8}$. Note that $$M_C = \{ 14,
 46,
 62,
 78,
 94,
 126,
 142,
 158,
 174,
 206,
 222,
 238,
 254,
 286,
 302,
 318,
 334,
 366,
 382,
 398 \} \textrm{and }$$ $$m \cdot E = \{ 26,
 42,
 58,
 74,
 106,
 122,
 138,
 154,
 186,
 202,
 218,
 234,
 266,
 282,
 298,
 314,
 346,
 362,
 378,
 394\}.$$ Now let $E_m$ be the element of $E$ such that $E_m \cdot m \equiv 26 \pmod{400}$. We claim that $a=2a_1$ where $a_1 \equiv 11 \pmod{400}$, $a_1 \nmid m$ and $\left(\frac{-a}{p}\right)=1$ for all odd primes $p \vert m_1$ will suffice. Moreover, here $A$ can be chosen so that $A^2 \equiv 4 \pmod{400}$.

\end{itemize}
\item[(Case 3)] $m=5m_1$ odd with $m \not\equiv \pm 5 \pmod{25}$. As to be expected, we have subcases:
\begin{itemize}
\item[(Subcase 1)] $m_1 \equiv 1, 5 \pmod{8}$. Let $E_m \in E$ be such that $E_m \cdot m \equiv 5 \pmod{200}$. Let $a \equiv 1 \pmod{200}$, be prime with $a \nmid m$ and $\left(\tfrac{-a}{p}\right)=1$ for all primes $p \vert m_1$. We see that:
$$
1  =  \prod_{p|m_1} \left( \dfrac{-a}{p} \right)
   =  \left( \dfrac{m_1}{a} \right)
   =  \left( \dfrac{-5m}{a} \right).
$$
Then choose $A$ so that $A^2 \equiv 4 \pmod{200}$.
\item[(Subcase 2)] $m_1 \equiv 3,7 \pmod{8}$. Note also that $m_1 \equiv 2, 3 \pmod{5}$. Let $E_m \in E$ be such that $E_m \cdot m \equiv 5 \pmod{200}$. Let $a=5a_1$ where $a_1 \equiv 1 \pmod{200}$ is prime with $a_1 \nmid m$ and $\left(\tfrac{-a}{p}\right)=1$ for all primes $p \vert m_1$. Choosing $A$ such that $A^2 \equiv 0 \pmod{200}$ gives the desired result.
\end{itemize}
\item[(Case 4)] $m=5\cdot 2 \cdot m_1$ with $m \not\equiv \pm 5 \pmod{25}$. Again, we have two subcases:
\begin{itemize}
\item[(Subcase 1)] $m =10m_1$ and $m_1 \equiv 1,5 \pmod{8}$. Let $E_m \in E$ be such that $E_m \cdot m \equiv 10 \pmod{200}$. Let $a \equiv 1 \pmod{200}$, $a \nmid m$ be prime with $\left(\tfrac{-a}{p}\right)=1$ for all $p \vert m_1$. We see that:
$$
1  =  \prod_{p|m_1} \left( \dfrac{-a}{p} \right)
   =  \prod_{p|m_1} \left( \dfrac{a}{p} \right)
   =  \prod_{p|m_1} \left( \dfrac{p}{a} \right)
   =  \left( \dfrac{m_1}{a} \right)
   =  \left( \dfrac{-2m_1}{a} \right)
   =  \left( \dfrac{-5m}{a} \right).
$$
Then, let $A$ be such that $A^2 \equiv 9 \pmod{200}$.
\item[(Subcase 2)] $m=10m_1$ and $m_1 \equiv 3,7 \pmod{8}$. Let $E_m \in E$ be such that $E_m \cdot m \equiv 30 \pmod{200}$. Let $a \equiv 21 \pmod{200}$, $a \nmid m$ be prime with $\left(\tfrac{-a}{p}\right)=1$ for all $p \vert m_1$. We see that:
$$
1  =  \prod_{p|m_1} \left( \dfrac{-a}{p} \right)
   =  (-1) \prod_{p|m_1} \left( \dfrac{a}{p} \right)
   =  (-1) \prod_{p|m_1} \left( \dfrac{p}{a} \right)
   =  (-1) \left( \dfrac{m_1}{a} \right)
   =  \left( \dfrac{-2m_1}{a} \right)
   =  \left( \dfrac{-5m}{a} \right).
$$
Then, let $A$ be such that $A^2 \equiv 9 \pmod{200}$.
\end{itemize}
\end{itemize}

Note that we did not consider any case where $m\equiv0 \pmod4$, i.e. $m = 4^k m_1$ for some $k,m_1\in\mathbb{Z}$ where $4\nmid m_1$. This omission is acceptable since, if $Q_2$ represents $m_1$, then $Q_2$ clearly represents $m$ as well. If $Q_2$ does not represent $m_1$ then $m_1\equiv{\pm 5} \pmod{25}$ so $m\equiv{\pm 5} \pmod{25}$ and $Q_2$ does not represent $m$. This completes the proof of Theorem \ref{D5}.

\section{{Forms of Determinant Six}}
We now show that that $x^2+y^2+6z^2$ represents all integers $m\neq 9^k(9\ell+3)$ for $k,\ell \in \mathbb Z^{\geq 0}$. Let $m \in \mathbb{Z}$, where $9 \nmid m$ and with $m \not\equiv 3 \pmod{9}$. We refer the reader to the series of lemmas appearing in the proof of representation by $x^2+2y^2+3z^2$, as it is the other form of determinant $6$. We maintain notation and setup. Again choose $b \equiv h \equiv 0 \pmod{6m}$.
Consider
$$A^2+a \equiv B^2+b \equiv 2AB+2h \equiv  0 \pmod{m}.$$
The second and third congruence conditions are satisfied if we choose $B \equiv 0 \pmod{m}$. In addition to securing the first congruence condition is satisfied, we also enforce $\tfrac{A^2+a}{m} \equiv 10 \pmod{16}$. Subsequently we have three different cases.
\begin{itemize}
\item[(Case 1)]
$m \equiv 1 \pmod{4}$.
\begin{itemize}
\item[(Subcase 1)]
Assume $3 \nmid m$. Choose $a \equiv 1 \pmod{48}$ an odd prime, $a \nmid m$, and $(\tfrac{-a}{p})=1$ for all odd primes $p|m$. We see that:
$$
1  =  \prod_{p|m} \left(\dfrac{-a}{p}\right)
   =  \prod_{p|m} \left(\dfrac{p}{a}\right)
   =  \left(\dfrac{m}{a}\right)
   =  \left(\dfrac{-6m}{a}\right).
$$
Note we also have $\tfrac{A^2+a}{m} \equiv 10 \pmod{16}$. If $m \equiv 1 \pmod{8}$, then chose $A^2 \equiv 9 \pmod{16}$. Else, choose $A^2 \equiv 1 \pmod{16}$.
\item[(Subcase 2)]
$m \equiv 33 \pmod{36}$.
Let $m=3m_1$ for $m_1 \equiv 11 \pmod{12}$ and let $a := 3a_1$ where $a_1 \equiv 3 \pmod{8}$ is an odd prime with $a_1 \nmid m$ so that $(\tfrac{-a}{p})=1$ for all primes $p|m_1$. The remainder of the proof follows similar to Subcase 1; when $m \equiv 1 \pmod{8}$, then $a_1 \equiv 3 \pmod{16}$ and when $m \equiv 5 \pmod{8}$ let $a_1 \equiv 11 \pmod{16}$.
\end{itemize}
\item[(Case 2)]
$m \equiv 3 \pmod{4}$.
\begin{itemize}
\item[(Subcase 1)]
Assume $3 \nmid m$. Choose $a \equiv 13 \pmod{24}$ a prime, with $a \nmid m$ and $(\tfrac{-a}{p})=1$ for all primes $p|m$. 
All conditions are satisfied, including $\tfrac{A^2+a}{m} \equiv 10 \pmod{16}$. If $m \equiv 3 \pmod{8}$, then require $a \equiv 5 \pmod{16}$. If $m \equiv 7 \pmod{8}$, require $a \equiv13 \pmod{16}$.
\item[(Subcase 2)]
$m=3m_1 \equiv 15 \pmod{36}$.
Let $m=3m_1$ for $m_1 \equiv 5 \pmod{12}$ and let $a := 3a_1$ where $a_1 \equiv 7 \pmod{8}$ is a prime with $a_1 \nmid m$ and $(\tfrac{-a}{p})=1$ for all primes $p|m_1$. 
We can again choose $A^2$ so that $\tfrac{A^2+a}{m} \equiv 10 \pmod{16}$, as regardless of $m \equiv 3,7 \pmod{8}$, $a=3a_1 \equiv 5,13 \pmod{16}$.
\end{itemize}
\item[(Case 3)]
$m \equiv 2 \pmod{4}$ and $3 \nmid m$. We write $m=2m_1$ where $m_1 \equiv 1,3 \pmod{4}$. We choose $a \equiv 19 \pmod{24}$ a prime with $a \nmid m$ and $(\tfrac{-a}{p})=1$ for all primes $p \vert m_1$ to ensure that $$1= \left( \dfrac{-a}{p} \right) = \left(\dfrac{-6m}{a} \right).$$ To ensure that $A$ can be chosen such that $\tfrac{A^2+a}{m} \equiv 10 \pmod{16}$ we have (minor) cases on $m \pmod{32}$:
\begin{itemize}
\item[(Subcase 1)] $m \equiv 2, 18 \pmod{32}$. Then $10m \equiv 20 \pmod{32}$. Above the only restriction on $a$ was that $a \equiv 3 \pmod{8}$ which implies $a \equiv 3,11,19,27 \pmod{32}$. For each of these choices of lifts, a choice of $A$ exists (respectively, allow $A^2 \equiv 17,9,1,25 \pmod{32}$).
\item[(Subcase 2)] $m \equiv 6, 22 \pmod{32}$. Then $10m \equiv 28 \pmod{32}$. Again, we see that $a \equiv 3,11,19,27 \pmod{32}$ each yields an appropriate choice of $A^2$.
\item[(Subcase 3)] $m \equiv 10,26 \pmod{32}$. Here $10m \equiv 4 \pmod{32}$, and for each possible value of $a \pmod{32}$ there is a corresponding choice of $A$ satisfying $A^2+a \equiv 10m \pmod{32}$.
\item[(Subcase 4)] $m \equiv 14, 30 \pmod{32}$. This means $10m \equiv 12 \pmod{32}$ and for each possible value of $a \pmod{32}$ there exists a choice of $A$ satisfying $A^2+a \equiv 10m \pmod{32}$.
\end{itemize}
\item[(Case 4)] $m \equiv 2 \pmod{4}$ and $3 \vert m$. Here we write $m=6m_1$ where $m_1 \equiv 1,7 \pmod{12}$. We let $a:= 3a_1$ where $a_1 \equiv 1 \pmod{8}$ is an odd prime with $a_1 \nmid m$ so that $(\tfrac{-a}{p})=1$ for all primes $p \vert m_1$. This allows
$$
1  =  \prod_{p|m_1} \left(\dfrac{-a}{p}\right)
   =  \prod_{p|m_1} \left(\dfrac{a_1}{p}\right)
   =  \left(\dfrac{m_1}{a_1}\right)
   =  \left(\dfrac{-6m}{a_1}\right).
$$
Moreover, for each of the possibilities of $m \pmod{32}$, we have behavior identical to the previous subcases. 
\end{itemize}
Note that we did not consider any case where $m \equiv 0 \pmod4$, i.e. $m = 4^km_1$ for some $k,m_1 \in \mathbb{Z}$ where $4 \nmid m_1$. This is an acceptable omission. If $Q_1$ represents $m_1$, then $Q_1$ represents $m$ as well. If $Q_1$ does not represent $m_1$ then $m_1 \equiv 3 \pmod9$ so $m \equiv 3 \pmod9$ and $Q_1$ does not represent $m$. This completes the proof of Theorem \ref{D6}.

\end{document}